\documentclass[psamsfonts,fceqn,leqno]{amsart}
  \usepackage{amsthm}
  \usepackage{amsmath}
  \usepackage{amssymb}
  \usepackage{mathrsfs}
  
  \newtheorem{theorem}{Theorem}[section]
  \newtheorem{corollary}[theorem]{Corollary}
  \newtheorem{proposition}[theorem]{Proposition}

  \theoremstyle{definition}
  \newtheorem{definition}[theorem]{Definition}

  \numberwithin{equation}{section}

  \title[Aggregate Preferred Correspondence and
  the Existence of a MREE]
  {Aggregate Preferred Correspondence and \\ the
  Existence of a Maximin REE}

  \author[A. Bhowmik]{Anuj Bhowmik}
  \address{School of Computing and Mathematical Sciences,
  Auckland University of Technology, Private Bag 92006, Auckland
  1142, New Zealand}
  \email{anuj.bhowmik@aut.ac.nz}

  \author[J. Cao]{Jiling Cao}
  \address{School of Computing and Mathematical Sciences,
  Auckland University of Technology, Private Bag 92006, Auckland
  1142, New Zealand}
  \email{jiling.cao@aut.ac.nz}

  \author[N.~C. Yannelis]{Nicholas C. Yannelis}
  \address{Department of Economics, Henry B. Tippie College of
  Business, The University of Iowa, IA 52242-1994, USA, and
  Economics - School of Social Sciences, The University of
  Manchester, Manchester M13 9PL, UK}
  \email{nicholasyannelis@gmail.com}

  \thanks{\hspace{-1.66em} \emph{JEL Classification Numbers:}
  D51, D82.}

\begin{document}

  \begin{abstract}
  In this paper, a general model of a pure exchange differential
  information economy is studied. In this economic model, the space
  of states of nature is a complete probability measure space, the
  space of agents is a measure space with a finite measure, and
  the commodity space is the Euclidean space. Under appropriate
  and standard assumptions on agents' characteristics, results
  on continuity and measurability of the aggregate preferred
  correspondence in the sense of Aumann in \cite{Aumann:66}
  are established. These results together with other techniques
  are then employed to prove the existence of a maximin rational
  expectations equilibrium (maximin REE) of the economic model.
  \end{abstract}

  \maketitle

  \section{Introduction} \label{sec:intro}

  When traders come to a market with different information about the
  items to be traded, the resulting market prices may reveal to some
  traders information originally available only to others. The possibility
  for such inferences rests upon traders having expectations of how
  equilibrium prices are related to initial information. This endogenous
  relationship was considered by Radner in his seminal paper
  \cite{Radner:79}, where he introduced the concept of a rational
  expectations equilibrium by imposing on agents the Bayesian (subjective
  expected utility) decision doctrine. Under the Bayesian decision
  making, agents maximize their subjective expected utilities
  conditioned on their own private information and also on the
  information that the equilibrium prices generate. The resulting
  equilibrium allocations are measurable with respect to the private
  information of each individual and also with respect to the information
  the equilibrium prices generate and clear the market for every state
  of nature. In both papers \cite{Allen:81} and \cite{Radner:79},
  conditions on the existence of a Bayesian rational expectations
  equilibrium (REE) were studied and some generic existence results
  were proved. However, Kreps \cite{Kreps:77} provided an example that
  shows that the Bayesian REE may not exist universally. In addition,
  a Bayesian REE may fail to be fully Pareto optimal and incentive
  compatible and may not be implementable as a perfect Bayesian
  equilibrium of an extensive form game, refer to Glycopantis et al.
  \cite{Glycopantis:2005} for more details.

  \medskip
  In a recent paper \cite{de Castro-Pesce-Yannelis:11}, de Castro et al.
  introduced a new notion of REE by a careful examination of Krep's
  example of the nonexistence of the Bayesian REE. In this formulation,
  the Bayesian decision making adopted in the papers of Allen \cite{Allen:81} and Radner \cite{Radner:79} was abandoned and
  replaced by the maximin expected utility (MEU) (see Gilboa and
  Schmeidler \cite{Gilboa:89}). In this new setup, agents maximize
  their MEU conditioned on their own private
  information and also on the information the equilibrium prices have
  generated. Contrary to the Bayesian REE, the resulting maximin
  REE may not be measurable with respect to the private information
  of each individual or the information that the equilibrium prices generate.

  \medskip
  Although Bayesian REE and maximin REE coincide in some special cases
  (e.g., fully revealing Bayesian REE and maximin REE), these two
  concepts are in general not equivalent. Nonetheless, the introduction
  of the MEU into the general equilibrium modeling enables de Castro
  et al. to prove that the maximin REE exists universally under the
  standard continuity and concavity assumptions on the utility
  functions of agents.
  Furthermore, they showed that the maximin REE is incentive compatible
  and efficient. Note that in the economic model considered in
  \cite{de Castro-Pesce-Yannelis:11}, it is assumed that there are
  finitely many states of nature and finitely many agents,
  and the commodity space is finite-dimensional. Thus, one of open
  questions is whether their existence theorem can be extended to more
  general cases. The main motivation of this paper is to tackle this
  question. In this paper, a general model of a pure exchange differential
  information economy is studied. In this economic model, the space
  of states of nature is a complete probability measure space,
  and the space of
  agents is a measure space with a finite measure. Under appropriate
  and standard assumptions on agents' characteristics, the existence
  of a maximin rational expectations equilibrium (maximin REE) of this
  general economic model is established.

  \medskip
  This paper is organized as follows. In Section \ref{sec:model},
  the economic model, some notation and assumptions are introduced
  and explained. Section \ref{sec:aggregate}, some results on
  continuity and measurability of agents' aggregate preferred
  correspondence in the sense of Aumann in \cite{Aumann:66} are
  established. These results are key techniques employed to prove
  the existence of a maximin rational expectations equilibrium
  of the economic model in Section \ref{sec:existence}. In Section
  \ref{sec:conclusion}, results and techniques appeared in this
  paper are compared with relevant results in the literature.
  Finally, some mathematical preliminaries and key facts used in
  this paper are presented in the appendix appeared at the end
  of the paper.

  \section{Differential information economies} \label{sec:model}

  In this paper, a model of a pure exchange economy $\mathscr E$
  with differential information is considered. The space of states
  of nature is a complete probability measure space $(\Omega,
  \mathscr F, \nu)$. The space of agents is a measure space
  $(T, \Sigma, \mu)$ with a finite measure $\mu$. The commodity
  space is the $\ell$-dimensional Euclidean space ${\mathbb
  R}^\ell$, and the positive cone ${\mathbb R}^\ell_+$ is the
  \emph{consumption set} for each agent $t\in T$ in every state
  of nature $\omega \in \Omega$. Each agent $t\in T$ is
  associated with her/his \emph{characteristics} $({\mathscr
  F}_t, U(t,\cdot, \cdot), a(t,\cdot), q_t)$, where ${\mathscr
  F}_t$ is the $\sigma$-algebra generated by a partition $\Pi_t$
  of $\Omega$ representing the \emph{private information} of
  $t$; $U(t,\cdot,\cdot):\Omega \times {\mathbb R}^\ell_+ \to
  \mathbb R$ is a \emph{random utility function} of $t$; $a(t,
  \cdot): \Omega\rightarrow \mathbb R^\ell_+$ is the \emph{random
  initial endowment} of $t$ and $q_t$ is a probability measure
  on $\Omega$ giving the \emph{prior} of $t$. The economy extends
  over two time periods $\tau = 1, 2$. At the ex ante stage
  ($\tau=0$), only the above description of the economy is a
  common knowledge. At the stage $\tau=1$, agent $t$ only knows
  that the realized state of nature belongs to the event
  ${\mathscr F}_t(\omega^\ast)$, where $\omega^\ast$ is the
  true state of nature at $\tau=2$. With this information (or
  with the information acquired through prices), agents trade.
  At the ex post stage ($\tau=2$), agents execute the trades
  according to the contract agreed at period $\tau=1$, and
  consumption takes place. The coordinate-wise order on
  ${\mathbb R}^\ell$ is denoted by $\leq$
  and the symbol $x\gg 0$ means that $x$ is an interior point
  of ${\mathbb R}^\ell_+$, and $x>0$ means that $x\ge 0$ but
  $x\ne 0$. Let $L_1 \left(\mu,{\mathbb R}^\ell_+\right)$ be
  the set of Lebesgue integrable functions from $T$ to
  ${\mathbb R}^\ell_+$. An \emph{allocation} in $\mathscr E$
  is a function $f: T\times \Omega\rightarrow {\mathbb
  R}^\ell_+$ such that $f(\cdot,\omega)\in L_1 \left(\mu,
  {\mathbb R}^\ell_+\right)$ for all $\omega \in\Omega$. An
  allocation $f$ in $\mathscr E$ is \emph{feasible} if
  \[
  \int_T f(\cdot, \omega) d\mu = \int_T a(\cdot, \omega)
  d\mu
  \]
  for all $\omega \in \Omega$.

  \bigskip
  The following standard assumptions on agents' characteristics
  shall be used:

  \medskip
  \noindent
  (${\bf A}_1$) The initial endowment function $a: T\times
  \Omega\rightarrow \mathbb R^\ell_+$ is jointly measurable
  such that $\int_T a(\cdot, \omega)d\mu\gg 0$ for each $\omega
  \in\Omega$.

  \medskip
  \noindent
  (${\bf A}_2$) $U(\cdot,\cdot, x)$ is jointly measurable
  for all $x \in \mathbb R^\ell_+$ and $U(t,\omega, \cdot)$
  is continuous for all $(t, \omega)\in T\times \Omega$.

  \medskip
  \noindent
  (${\bf A}_3$) For each $(t, \omega) \in T\times \Omega$,
  $U(t, \omega, \cdot)$ is monotone in the sense that if $x,
  y\in\mathbb R^\ell_+$ with $y> 0$, then $U(t,\omega, x+y)
  >U(t,\omega, x)$.

  \medskip
  \noindent
  (${\bf A}_4$) For each $(t, \omega) \in T\times \Omega$,
  $U(t, \omega, \cdot)$ is concave.

  \medskip
  Let
  \[
  \Delta= \left\{p \in \mathbb R^\ell_+: p\gg 0 \mbox{ and }
  \sum_{h= 1}^\ell p^h=1\right\}.
  \]
  Then, each element $p\in \Delta$ is viewed as a (normalized)
  \emph{price system} in the deterministic case. The
  \emph{budget correspondence}
  $B:T\times \Omega\times \Delta \rightrightarrows \mathbb
  R^\ell_+$ is defined by
  \[
  B(t, \omega, p)= \left\{x\in \mathbb
  R^\ell_+: \langle p, x\rangle\le \langle p, a(t,\omega)
  \rangle\right\}
  \]
  for all $(t,\omega,p) \in T \times \Omega \times \Delta$.
  Obviously, $B$ is non-empty and closed-valued. For each $\omega
  \in \Omega$, by Theorem 2 in \cite[p.151]{Hildenbrand:74},
  there are $p(\omega)\in \Delta$ and an allocation
  $f(\cdot, \omega)$ such that $(f(\cdot, \omega), p(\omega))$
  is a Walrasian equilibrium of the deterministic economy
  $\mathscr E(\omega)$, given by
  \[
  \mathscr E(\omega) = ((T,\Sigma, \mu); \mathbb R^\ell_+;
  (U(t, \omega, \cdot), a(t,\omega)): t\in T).
  \]
  Define a function $\delta: \Delta\to \mathbb R_+$ by
  \[
  \delta(p)= \min\left\{p^h: 1\le h \le\ell \right\},
  \]
  where $p= (p^1,..., p^\ell)\in \Delta$. For any $(t, \omega,
  p)\in T\times \Omega\times \Delta$, let
  \[
  \gamma(t, \omega, p)= \frac{1}
  {\delta(p)}\sum_{h=1}^\ell a^h(t, \omega)
  \]
  and
  \[
  b(t, \omega,
  p)= (\gamma(t, \omega, p), ..., \gamma(t, \omega, p)).
  \]
  Define $X: T\times \Omega\times \Delta
  \rightrightarrows \mathbb R^\ell_+$ by
  \[
  X(t,\omega,p)=\{x\in \mathbb R^\ell_+: x\le b(t,\omega,p)\}
  \]
  for all $(t, \omega, p)\in
  T\times \Omega\times \Delta$. Note that $X$ is non-empty,
  compact- and convex-valued such that $B(t, \omega, p)
  \subseteq X(t, \omega, p)$ for all $(t, \omega, p)\in T
  \times \Omega\times \Delta$. It can be readily verified
  that for every $(t, \omega) \in T \times \Omega$, the
  correspondence $X(t,\omega, \cdot): \Delta \rightrightarrows
  {\mathbb R}^\ell_+$ is Hausdorff continuous. Define two
  correspondences $C,\ C^{X}: T\times
  \Omega\times\Delta\rightrightarrows \mathbb R^\ell_+$ by
   \[
  C(t, \omega, p)= \left\{y\in \mathbb R^\ell_+: U(t,
  \omega, y)\ge U(t, \omega, x) \mbox{ for all } x\in
  B(t, \omega, p)\right\}
  \]
  and
  \[
  C^X(t, \omega, p)= C(t, \omega, p)\cap X(t, \omega, p).
  \]
  Obviously,
  \[
  B(t, \omega, p)\cap C(t, \omega, p)= B(t, \omega, p)\cap
  C^X (t, \omega, p)
  \]
  holds for all $(t, \omega, p)\in T\times \Omega \times
  \Delta$. Note that under (${\bf A}_2$), $U(t, \omega, \cdot)$
  is continuous on the non-empty compact set $B(t,\omega, p)$.
  Thus, one has
  \[
  B(t, \omega, p)\cap C(t, \omega, p) \ne \emptyset
  \]
  for all $(t, \omega, p)\in T\times \Omega \times \Delta$.

  \begin{proposition} \label{prop:implication}
  Let $(t,\omega, p) \in T\times \Omega \times \Delta$.
  Under {\rm (${\bf A}_3$)}, $\langle p, x\rangle\ge
  \langle p, a(t,\omega)\rangle$ for every point $x \in
  C^X(t, \omega, p)$.
  \end{proposition}

  \begin{proof}
  Assume that $\langle p, x_0\rangle<\langle p, a(t,\omega)
  \rangle$ for some point $x_0 \in C^X(t, \omega, p)$. Then,
  one can choose some $y \in \mathbb R^\ell_+$ such that
  $y>0$ and $\langle p, x_0+y \rangle<\langle p, a(t,\omega)
  \rangle$. Thus, $x_0 + y \in B(t, \omega, p)$. Since $x_0
  \in C^X(t, \omega, p)$, one has $U(t,\omega, x_0) >
  U(t, \omega, x_0+y)$. However, (${\bf A}_3$) implies
  $U(t, \omega, x_0+y) > U(t, \omega,x_0)$. This is a
  contradiction, which completes the proof.
  \end{proof}

  Following Aumann in \cite{Aumann:66}, $C^X(t,\omega, p)$
  is called the \emph{preferred set} of agent $t$ at the
  price $p$ and state of nature $\omega$, and $\int_T C^{X}
  (\cdot,\omega, p)d\mu$ is called the \emph{aggregate preferred
  set} at the price $p$ and state of nature $\omega$. Moreover,
  we shall call
  \[
  \int_T C^X(\cdot, \cdot,\cdot)d\mu:
  \Omega\times \Delta \rightrightarrows \mathbb R^\ell_+
  \]
  the \emph{aggregate preferred correspondence}.
  In the next section, we shall discuss some descriptive
  properties of this object. These properties will
  be used to derive the existence of a maximin rational
  expectations equilibrium of our economic model.

  \section{Properties of the aggregate preferred correspondence}
  \label{sec:aggregate}

  In this section, some results on continuity and measurability
  of the aggregate preferred correspondence are presented.

  \begin{proposition}\label{prop:budcaratheodory}
  Under \emph{({\bf A$_1$})}, for every $(\omega, p)\in \Omega\times
  \Delta$, $B(\cdot, \omega, p): T \rightrightarrows \mathbb
  R^\ell_+$ and  $X(\cdot, \omega, p): T \rightrightarrows
  \mathbb R^\ell_+$ are lower measurable.
  \end{proposition}

  \begin{proof}
  Here, only the proof of lower measurability of
  $B(\cdot, \omega, p)$ is provided. The other case can be done
  analogously. Fix $(\omega, p)\in \Omega\times \Delta$. Define
  a function
  $h: T\times \mathbb R^\ell_+\to \mathbb R$ by letting
  \[
  h(t,x)=\langle p, x\rangle-\langle p, a(t, \omega)\rangle
  \]
  for all $(t, x)\in T\times \mathbb R^\ell_+$. Then, $h(\cdot,
  x)$ is measurable for all $x\in \mathbb R^\ell_+$. Note
  that
  \[
  B(t,\omega, p)= h(t, \cdot)^{-1}((-\infty, 0]).
  \]
  Let $V \subseteq \mathbb R^\ell_+$ be a non-empty open
  subset, and put
  \[
  V\cap \mathbb Q_+^\ell= \{x_k: k \ge 1\}.
  \]
  It is worth to point out that if $x\in B(t, \omega, p)\cap V$,
  then $x_k \in B(t, \omega, p)$ for some $k\ge 1$. Since
  $h(\cdot,x_k)$ is measurable, one has
  \[
  \left\{t\in T:h(t, x_k)\in
  (-\infty, 0]\right\} \in \Sigma
  \]
  for all $k\ge 1$. Thus,
  \begin{eqnarray*}
  B(\cdot, \omega, p)^{-1}(V)&=& \bigcup_{k\ge 1}\left
  \{t\in T: x_k\in B(t, \omega, p)\right\}\\
  &=& \bigcup_{k\ge 1}\left\{t\in T: h(t, x_k)\in
  (-\infty, 0]\right\}
  \end{eqnarray*}
  belongs to $\Sigma$. It follows that $B(\cdot, \omega, p)$
  is lower measurable.
  \end{proof}

  \begin{proposition}\label{prop:Clowermeasure}
  Under \emph{({\bf A$_1$})-({\bf A$_3$})}, $\int_T C^{X}(\cdot,
  \cdot,\cdot)d\mu$ is non-empty compact-valued.
  \end{proposition}

  \begin{proof}
  Fix $(\omega, p)\in \Omega\times \Delta$. By ({\bf A$_2$}),
  $C^{X}(t,\omega, p)$ is non-empty closed for all $t\in T$. By
  the lower measurability of $B(\cdot, \omega, p)$, there
  exists a sequence $\{f_n: n \ge 1\}$ of measurable
  functions from $T$ to ${\mathbb R}^\ell_+$ such that
  \[
  B(t, \omega,p)= \overline {\{f_n(t):
  n \ge 1\}}
  \]
  for all $t\in T$. For each $n \ge 1$, define a
  correspondence $C_n: T \rightrightarrows \mathbb R^\ell_+$ by
  letting
  \[
  C_n(t)= \left\{x\in \mathbb R^\ell_+: U(t, \omega, x)
  \ge U(t, \omega, f_n(t))\right\}
  \]
  for all $t\in T$. Obviously, one has
  \[
  C(t,\omega, p)\subseteq \bigcap_{n\ge 1}C_n(t)
  \]
  for all $t \in T$. If $x\in {\mathbb R}^\ell_+ \setminus
  C(t,\omega, p)$ for some $t\in T$, there exists a point
  $y\in B(t,\omega,p)$ such that $U(t, \omega, y)>
  U(t, \omega, x)$. By ({\bf A$_2$}), there
  exists an $n_0 \ge 1$ such that $U(t, \omega,f_{n_0}(t))
  > U(t, \omega, x)$. This implies $x\notin C_{n_0} (t)$.
  Thus, it is verified that
  \[
  C(t,\omega, p)= \bigcap_{n\ge 1}C_n(t)
  \]
  for all $t\in T$. Fix $n\ge 1$, and define a function
  $h:T\times \mathbb R^\ell_+\to \mathbb R$ by
  \[
  h(t, x)=U(t, \omega, f_n(t))- U(t, \omega, x).
  \]
  Clearly, $h$ is Carath\'{e}odory. Similar to  Proposition
  \ref{prop:budcaratheodory}, one can show that $C_n$ is lower
  measurable. Since $X(\cdot,\omega, p)$ is
  compact-valued and
  \[
  C^X(\cdot,\omega,p)=
  \bigcap_{n\ge 1}C_n(\cdot) \cap X(\cdot, \omega, p),
  \]
  then $C^{X}(\cdot, \omega, p)$ is lower measurable.
  By the Kuratowski-Ryll-Nardzewski measurable selection theorem
  in \cite{Kuratowski:65},
  $C^X (\cdot, \omega, p)$ has a measurable selection
  which is also integrable, as $b(\cdot, \omega, p)$
  is so. Since $C^X(\cdot, \omega, p)$ is closed-valued and
  integrably bounded, $\int_T C^{X}(\cdot, \omega, p)d\mu$ is
  compact-valued.
  \end{proof}

  Next, we establish Hausdorff continuity of the aggregate
  preferred correspondence with respect to the variable $p
  \in \Delta$.

  \begin{theorem}\label{thm:contintC}
  Assume \emph{({\bf A$_1$})-({\bf A$_3$})}. For each $\omega
  \in \Omega$, $\int_T C^X(\cdot, \omega, \cdot)d\mu: \Delta
  \rightrightarrows {\mathbb R}^\ell_+$ is Hausdorff continuous
  \end{theorem}

  \begin{proof}
  Fix $\omega\in \Omega$. Let $\{p_n: n\ge 1\}\subseteq \Delta$
  converge to $p\in \Delta$.
  Choose $\varepsilon> 0$ and $N\ge 1$ such that $\varepsilon<
  \delta(p)$ and $\varepsilon<\delta(p_n)$ for all $n\ge N$. Let
  \[
  c= \min\left\{\delta(p_n), \varepsilon: n= 1, 2, \cdots,
  N- 1\right\},
  \]
  \[
  d(t,\omega)= \frac{1}{c}\sum_{h= 1}^\ell a^h (t,\omega),
  \]
  and
  \[
  \xi(t, \omega)= (d(t,\omega), ..., d(t,\omega)).
  \]
  Define $M(\omega)$ by
  \[
  M(\omega) =\left\{ x \in \mathbb R^\ell_+: x\le \int_T
  \xi (\cdot, \omega) d\mu\right\}.
  \]
  Since $X(\cdot,\omega, p_n)$ and $X(\cdot,\omega, p)$ are
  upper bounded by $\xi(\cdot,\omega)$, then $\int_T C^{X}(\cdot,
  \omega, p_n)d\mu$ and $\int_T C^{X}(\cdot, \omega, p)d\mu$
  are contained in the compact subset $M(\omega)$ of $\mathbb
  R^\ell_+$ . Thus, one only needs to show that $\left\{
  \int_T C^{X}(\cdot, \omega, p_n)d\mu: n\ge 1\right\}$ converges
  to $\int_T C^{X}(\cdot, \omega, p)d\mu$ in the Hausdorff metric
  topology on ${\mathscr K}_0(M(\omega))$, which is equivalent to
  \[
  {\rm Li} \int_T C^{X}(\cdot, \omega, p_n)d\mu =
  {\rm Ls} \int_T C^{X}(\cdot, \omega, p_n)d\mu =
  \int_T C^{X}(\cdot, \omega, p)d\mu.
  \]
  The verification of the above equation can be split into two
  steps. First, one verifies
  \[
  {\rm Ls}\int_T C^{X}(\cdot,
  \omega, p_n)d\mu \subseteq \int_T C^{X}(\cdot, \omega,
  p)d\mu.
  \]
  To do this, it is enough to verify that for any $t\in T$,
  \[
  {\rm Ls} C^{X}(t, \omega, p_n)\subseteq C^{X}(t, \omega, p).
  \]
  Pick $t\in T$ and $x\in {\rm Ls} C^{X}(t,\omega, p_n)$.
  Then, there exist positive integers $n_1<n_2<n_3 <\cdots$ and
  for each $k$ a point $x_k \in C^{X}(t, \omega, p_{n_k})$ such
  that $\{x_k: k\ge 1\}$ converges to $x$. It is obvious that
  $x\in X(t, \omega, p)$. If $x\notin C^{X}(t,\omega, p)$, by
  the continuity of $U(t,\omega,\cdot)$, one can choose some $y
  \in \mathbb R^\ell_+$ such that $\langle p, y\rangle<\langle
  p, a(t,\omega)\rangle$ and $U(t,\omega,y)> U(t,\omega,x)$.
  By the Hausdorff continuity of $X(t,\omega, \cdot)$, $\{X(t,
  \omega, p_{n_k}): k\ge 1\}$ converges to $X(t,\omega, p)$
  in the Hausdorff metric topology. Since $y \in X(t,\omega,
  p)$, there exists a sequence $\{y_k: k\ge 1\}$ such that
  $y_k\in X(t,\omega, p_{n_k})$ for all $k\ge 1$ and $\{y_k:
  k\ge 1\}$ converges to $y$. It follows that $U(t,\omega,
  y_k)> U(t,\omega, x_k)$ and $\langle p_{n_k}, y_k\rangle <
  \langle p_{n_k}, a(t, \omega)\rangle$ for all sufficiently
  large $k$, which is a contradiction with $x_k \in C^X
  \left(t,\omega, p_{n_k}\right)$ for all $k \ge 1$. Therefore,
  one must have $x\in C^{X}(t,\omega, p)$. Secondly, one needs
  to verify
  \[
  \int_T C^{X}(\cdot, \omega, p)d\mu \subseteq {\rm Li}
  \int_T C^{X}(\cdot, \omega, p_n)d\mu.
  \]
  It is enough to verify that for all $t\in T$,
  \[
  C^{X}(t, \omega, p) \subseteq {\rm Li} C^{X}(t, \omega, p_n).
  \]
  Fix $t\in T$ and pick $d\in C^{X}(t,\omega, p)$. If $d= b(t,
  \omega, p)$, then $b(t,\omega, p_n)\in C^X (t, \omega, p_n)$
  and the sequence $\{b(t,\omega,p_n): n\ge 1\}$ converges to
  $d$. Assume $d< b(t, \omega,p)$. Select $\delta> 0$ such that
  \[
  d+ (0,\cdots, \delta,\cdots, 0)\le b(t, \omega, p).
  \]
  Further, choose a sequence $\{\delta_i: i\ge 1\}$ in $(0,
  \delta]$ converging to $0$. For each $i\ge 1$, let
  \[
  d^i= d+ (0,\cdots, \delta_i,\cdots, 0),
  \]
  and choose a sequence $\{d_n^i: n\ge 1\}$
  such that for each $n$, $d_n^i \in X(t,\omega, p_n)$ and
  $\{d_n^i: n \ge 1\}$ converges to $d^i$. It is
  claimed that for each $i\ge 1$, $d_n^i\in C^{X}(t, \omega,
  p_n)$ for sufficiently large $n$. Otherwise, there must exist
  an $i_0$ and a subsequence $\{d_{n_k}^{i_0}: k\ge 1\}$
  of $\{d_{n}^{i_0}: n\ge 1\}$ such that $d_{n_k}^{i_0} \notin
  C^{X}(t, \omega, p_{n_k})$. Let $b_k\in B(t, \omega, p_{n_k})$
  and $U(t,\omega, b_k)> U(t,\omega, d_{n_k}^{i_0})$
  for all $k\ge 1$. Then $\{b_k: k\ge 1\}$ has a subsequence
  converging to some $b\in B(t, \omega, p)$. Applying ({\bf
  A$_2$}) and ({\bf A$_3$}), one can obtain
  \[
  U(t,\omega, b)\ge U(t,\omega,d^{i_0})> U(t,\omega, d),
  \]
  which contradicts with the fact $d\in C^{X}(t, \omega, p)$.
  To complete the proof, note that the previous claim implies that
  for each $i$, $\{{\rm dist}(d^i, C^X(t,\omega, p_n)): n\ge1\}$
  converges to $0$. Since $\{d^i: i\ge 1\}$ converges to $d$,
  one concludes that $\{{\rm dist}(d, C^X(t,\omega, p_n)): n
  \ge1\}$ converges to $0$. This means that $d \in {\rm Li}
  C^{X}(t, \omega, p_n)$.
  \end{proof}

  The next result is crucial for the existence theorem
  in Section \ref{sec:existence}. In its proof, the following
  characterization of lower measurability of correspondences
  in \cite{aubin:08} is used: A correspondence
  $F: (\Omega, {\mathscr F},\nu) \rightrightarrows \mathbb
  R^\ell_+$ is lower measurable if and only if for all
  $y\in \mathbb R^\ell_+$, ${\rm dist}
  (y, F(\cdot)): (\Omega, {\mathscr F},\nu) \to \mathbb R_+$ is
  a measurable function.

  \begin{theorem}\label{thm:intmeasure}
  Assume \emph{({\bf A$_1$})-({\bf A$_3$})}. For each $p\in
  \Delta$, $\int_T C^{X}(\cdot, \cdot,p)d\mu: \Omega
  \rightrightarrows \mathbb R^\ell_+ $ is lower measurable.
  \end{theorem}

  \begin{proof}
  Fix $p\in \Delta$. Since $a$ and $U$ are $\Sigma\otimes \mathscr F$-measurable and $\Sigma\otimes \mathscr F\otimes \mathfrak
  B(\mathbb R^\ell_+)$-measurable respectively, by the argument
  of a result in \cite[131]{Yosida:80}, there
  exist two sequences $\{a_n: n\ge 1\}$ and $\{\psi_n: n\ge 1
  \}$ of simple $\Sigma\otimes \mathscr F$-measurable and
  simple $\Sigma\otimes \mathscr F\otimes \mathfrak B(\mathbb R^\ell_+)$-measurable functions respectively such that
  $\{a_n: n\ge 1\}$ uniformly converges to $a$
  on $T\times \Omega$ and $\{\psi_n: n \ge 1\}$ uniformly
  converges to $U$ on $T\times \Omega
  \times \mathbb R^\ell_+$. For each $n\ge 1$, write $a_n$
  and $\psi_n$ as
  \[
  a_n= \sum_{i\ge 1} e_i \chi_{T_i^n\times \Omega_i^n}
  \]
  and
  \[
  \psi_n= \sum_{i\ge 1} v_i \chi_{T_i^n\times
  \Omega_i^n\times B_i^n},
  \]
  where $e_i\in \mathbb R^\ell_+$, $v_i\in \mathbb R$, and
  $\{T_i^n\times \Omega_i^n\times B_i^n: i\ge 1\}$ is a
  partition of $T\times \Omega\times \mathbb R^\ell_+$
  for all $n\ge 1$.
  Choose $N\ge 1$ such that $\|a_n- a\|_\infty < 1$ for all
  $n \ge N$. By the measurability of $a_n(\cdot, \omega)$,
  $a_n(\cdot, \omega)\in L_1 \left(\mu,\mathbb R^\ell_+
  \right)$ for all $\omega\in \Omega$ and all
  $n\ge 1$ (replacing $a_n$ for all $1\le n< N$ by some
  constant functions, if necessary). Let
  \[
  \gamma_n(t,\omega)= \frac{1}{\delta(p)}\sum_{h=1}^\ell
  a_n^h(t,\omega)
  \]
  and
  \[
  b_n(t, \omega)= (\gamma_n(t, \omega),\cdots, \gamma_n(t,
  \omega)).
  \]
  Define $X_n, B_n, C_n: T\times \Omega
  \rightrightarrows \mathbb R^\ell_+$ such that
  \[
  X_n(t, \omega)= \left\{x\in \mathbb R^\ell_+: x\le
  b_n(t, \omega)\right\},
  \]
  \[
  B_n(t, \omega)= \left\{x\in \mathbb R^\ell_+: \langle p,
  x\rangle\le \langle p, a_n(t, \omega)\rangle\right\}
  \]
  and
  \[
  C_n(t, \omega)= \left\{y\in \mathbb R^\ell_+:
  \psi_n(t,\omega, y)\ge \psi_n(t, \omega, x) \mbox{ for all }
  x\in B_n(t, \omega)\right\}.
  \]
  In addition, define $C_n^X: T\times \Omega \rightrightarrows
  \mathbb R^\ell_+$ such that for all $(t, \omega) \in T \times
  \Omega$,
  \[
  C_n^{X}(t, \omega)= \left(C_n(t, \omega) \cup \{b_n(t,\omega)\}
  \right) \cap X_n(t, \omega).
  \]
  For every $n\ge 1$, define the correspondence $H_n: (\Omega,
  \mathscr F, \nu) \rightrightarrows L_1\left(\mu, \mathbb
  R^\ell_+ \right)$ by letting $H_n(\omega)={\mathscr S}_{C_n^{X}
  (\cdot, \omega)}$. Obviously, $H_n(\omega)\ne \emptyset$ for
  all $\omega\in \Omega$.

  \medskip
  {\it Claim 1. For each $n\ge 1$, $H_n$ is lower measurable.}
  For convenience, let $\Theta: L_1\left(\mu, \mathbb R^\ell_+
  \right) \times \Omega \to \mathbb R_+$ be the
  function such that
  \[
  \Theta(g, \omega)= {\rm dist}(g, H_n (\omega))
  \]
  for all $g \in L_1 \left(\mu, \mathbb R^\ell_+\right)$ and
  $\omega \in \Omega$. To verify the claim, one needs
  to verify that for all  $g \in L_1 \left(\mu, \mathbb R^\ell_+
  \right)$, $\Theta(g,\cdot)$ is measurable. Since $\Theta (\cdot,
  \omega): L_1\left(\mu, \mathbb R^\ell_+\right) \to \mathbb
  R_+$ is norm-continuous, it suffices to show that
  $\Theta (g, \cdot): \Omega \to \mathbb R_+$ is measurable for
  every simple function $g= \sum_{j= 1}^r x_j \chi_{T_j}$, where
  $x_j\in \mathbb R^\ell_+$. To this end, consider the function
  $\Gamma: T\times \Omega \to \mathbb R_+$ such that
  \[
  \Gamma (t, \omega)= {\rm dist}\left(g(t), C_n^X(t, \omega)\right)
  \]
  for all $(t, \omega)\in T\times \Omega$. Since $\Gamma$ is constant
  on each $(T^n_i\cap T_j)\times \Omega_i^n$, it is jointly
  measurable. Note that
  \[
  \Gamma(t,\omega) \le \|g(t)-b_n(t,\omega)\|
  \]
  for all $(t,\omega) \in T\times \Omega$. This implies for all
  $\omega \in \Omega$, $\Gamma(\cdot, \omega)$ is integrable.
  Thus, $\Theta(g, \cdot)$ is measurable and the claim is
  verified if one can show for all $\omega\in \Omega$,
  \[
  \int_T \Gamma(\cdot, \omega)d\mu =\Theta(g, \omega).
  \]
  Assume that
  \[
  \int_T \Gamma(\cdot, \omega_0)d\mu <\Theta(g, \omega_0)
  \]
  for some $\omega_0 \in \Omega$. Pick $\varepsilon> 0$ such that
  \[
  \int_T \Gamma(\cdot,
  \omega_0)d\mu+ \varepsilon\mu(T)< \Theta(g,\omega_0).
  \]
  Further, pick $t\in T_i^n\cap T_j$ and $y_{(i,j)}\in
  C_n^X(t, \omega_0)$ so that
  \[
  \|x_j- y_{(i, j)}\|< \Gamma(t,\omega_0)+ \varepsilon.
  \]
  Then the function $\zeta:
  T\to \mathbb R^\ell_+$, defined by $\zeta(t)= y_{(i, j)}$
  if $t\in T_i^n\cap T_j$, belongs to $H_n(\omega_0)$ and
  \[
  \|g- \zeta\|_1< \int_T \Gamma(\cdot, \omega_0)d\mu+
  \varepsilon\mu(T),
  \]
  which is a contradiction.

  \medskip
  {\it Claim 2. The correspondence $\int_T C_n^{X} (\cdot,
  \cdot)d\mu: (\Omega, \mathscr F, \nu) \rightrightarrows
  \mathbb R^\ell_+$ is lower measurable}.
  Consider the function $\xi: L_1\left(\mu,\mathbb R^\ell_+
  \right)\to \mathbb R^\ell_+$ defined by $\xi(f)= \int_T
  f d\mu$ for all $f \in L_1\left(\mu,\mathbb R^\ell_+
  \right)$. Let $V$ be an open subset of $\mathbb R^\ell_+$.
  Note that
  \[
  \xi \circ H_n (\omega)= \int_T C_n^{X}(\cdot,\omega)d\mu
  \]
  for all $\omega\in \Omega$, and
  \[
  (\xi \circ H_n)^{-1} (V)
  =\{\omega\in \Omega: H_n (\omega)\cap \xi^{-1}(V)\neq
  \emptyset\}.
  \]
  Since $\xi$ is norm-continuous, by Claim 1, $(\xi \circ
  H_n)^{-1} (V) \in \mathscr F$. This verifies the claim.

  \bigskip
  {\it Claim 3. For each $\omega \in\Omega$,}
  \[
  {\rm Li}\int_T C_n^{X}(\cdot, \omega)d\mu =
  {\rm Ls} \int_T C_n^{X}(\cdot, \omega)d\mu=\int_T C^{X}
  (\cdot, \omega, p)d\mu.
  \]
  To see this, for each $\omega\in \Omega$, put
  \[
  \alpha(\cdot,\omega)= \sup \left\{b_1(\cdot, \omega),\cdots,
  b_{N-1}(\cdot, \omega), b(\cdot, \omega, p)+ \left
  (\frac{\ell}{\delta(p)},\cdots, \frac{\ell}{\delta(p)}
  \right)\right\}.
  \]
  Then, $C^{X}(\cdot, \omega, p)$ and all $C_n^{X}(\cdot,\omega)$
  are integrably bounded by $\alpha(\cdot, \omega)$. Now, it
  suffices to verify that for all $t\in T$,
  \[
  {\rm Ls} C_n^{X}(t, \omega)\subseteq C^{X}(t, \omega, p),
  \]
  and
  \[
  C^{X}(t, \omega, p) \subseteq {\rm Li} C_n^{X}(t, \omega).
  \]
  First, let $x\in {\rm Ls} C_n^{X}(t,\omega)$.
  If $x= b(t, \omega, p)$, then $\{b_n(t, \omega): n \ge 1\}$
  converges to $x$ and $b_n(t,\omega) \in C_n^X(\cdot,\omega)$
  for all $n\ge 1$. Otherwise,
  there exist positive integers $n_1<n_2  <n_3 <\cdots$ and for
  each $k$ a point $x_k\in C_{n_k}^{X}(t, \omega)$ such that
  $\{x_k: k\ge 1\}$ converges to $x$. Obviously, $x_k\neq
  b_{n_k}(t, \omega)$ for all sufficiently large $k$, and $x\in
  X(t,\omega, p)$. If $x\notin C^{X}(t,\omega,p)$, there exists
  some $y\in B(t, \omega, p)$ such that $U(t,\omega,y)> U(t,
  \omega, x)$. By the continuity of $U(t,\omega,\cdot)$, $y$
  can be chosen so that $\langle p, y\rangle<\langle p,
  a(t,\omega)\rangle$. Since $\{X_{n_k}(t,\omega): k\ge 1\}$
  converges to $X(t,\omega, p)$ in the Hausdorff metric
  topology, there exists a sequence $\{y_k: k\ge 1\}$ such
  that $y_k\in X_{n_k} (t,\omega)$ for all $k\ge 1$ and
  $\{y_k: k\ge 1\}$ converges to $y$. By the inequality
  \begin{eqnarray*}
  |U(t,\omega, x)- \psi_{n_k}(t, \omega, x_k)|
  &<& |U(t,\omega, x)- U(t,\omega, x_k)|\\
  &+& |U(t,\omega,x_k)- \psi_{n_k}(t,\omega, x_k)|,
  \end{eqnarray*}
  the continuity of $U(t,\omega, \cdot)$ and the uniform
  convergence of $\psi_{n_k}(t, \omega, \cdot)$ to
  $U(t,\omega, \cdot)$, one concludes that $\psi_{n_k}(t,
  \omega, y_k)> \psi_{n_k}(t, \omega, x_k)$ and $\langle
  p, y_k\rangle< \langle p, a_{n_k}(t, \omega)\rangle$ for
  sufficiently large $k$, which contradicts with the fact
  that $x_k\in C_{n_k}^{X}(t, \omega)$ for all $k\ge 1$.
  Hence, $x\in C^{X}(t, \omega, p)$. Now, let $d\in C^{X}
  (t,\omega, p)$. If $d= b(t, \omega, p)$, there is
  nothing to verify. Thus, $d\in {\rm Li}
  C_n^{X}(t, \omega)$. Assume $d< b(t, \omega, p)$. Similar
  to that in the proof of Theorem \ref{thm:contintC},
  one can show that $d\in {\rm Li}C_n^{X}(t, \omega)$.

  \medskip
  To complete the proof, for each $\omega \in \Omega$, put
  \[
  M(\omega)= \left\{x\in \mathbb R^\ell_+: x\le \int_T
  \alpha(\cdot, \omega)d\mu\right\}.
  \]
  Clearly,
  $\int_T C_n^{X}(\cdot, \omega)d\mu$ and $\int_T C^{X}(\cdot,
  \omega, p)d\mu$ are contained in the compact set $M(\omega)$.
  It follows from Claim 3 that $\left\{\int_T C_n^{X}(\cdot, \omega)
  d\mu: n \ge 1\right\}$ converges to $\int_T C^{X}(\cdot, \omega, p)
  d\mu$ in $M(\omega)$ in the Hausdorff metric topology. It is
  well known that a nonempty compact-valued correspondence is
  lower measurable if and only if it is measurable when viewed
  as a single-valued function whose range space is the space of
  nonempty compact sets endowed with the Hausdorff metric topology.
  By Claim 2, $\int_T C^{X}(\cdot, \cdot, p)d\mu$ is lower
  measurable.
  \end{proof}

  \begin{corollary} \label{coro:jointmea}
  Assume \emph{({\bf A$_1$})-({\bf A$_3$})}. Then $\int_T C^X
  (\cdot, \cdot, \cdot)d\mu:\Omega\times \Delta\rightarrow
  {\mathscr K}_0(\mathbb R^\ell_+)$ is a jointly measurable
  function, where ${\mathscr K}_0 (\mathbb R^\ell_+)$ is endowed
  with the Hausdorff metric topology.
  \end{corollary}

  \begin{proof}
  By Theorem \ref{thm:contintC}, for every $\omega \in \Omega$,
  $\int_T C^X (\cdot, \omega,\cdot)d\mu: \Omega \times \Delta
  \to {\mathscr K}_0\left(\mathbb R^\ell_+\right)$ is continuous. Furthermore, by Theorem \ref{thm:intmeasure}, for every $p\in
  \Delta$, the correspondence $\int_T C^{X} (\cdot, \cdot,p)
  d\mu: \Omega \rightrightarrows \mathbb R^\ell_+$ is lower
  measurable. Hence, for every $p\in \Delta$, $\int_T C^{X}
  (\cdot,\cdot,p)d\mu: \Omega \rightarrow {\mathscr K}_0
  (\mathbb R^\ell_+)$ is measurable. This means
  that $\int_T C^{X} (\cdot, \cdot, \cdot) d\mu:
  \Omega \times \Delta\rightarrow {\mathscr K}_0(\mathbb
  R^\ell_+)$ is Carath\'{e}odory, and therefore is jointly
  measurable.
  \end{proof}

  \section{The existence of a maximin REE}
  \label{sec:existence}

  A \emph{price system} of $\mathscr E$ is a measurable
  function $\pi: (\Omega, \mathscr F, \nu) \to \Delta$. Let
  $\sigma(\pi)$ be the smallest sub-algebra of $\mathscr F$ such
  that $\pi$ is measurable and let ${\mathscr G}_t={\mathscr F}_t
  \vee \sigma (\pi)$. For each $\omega\in \Omega$, let ${\mathscr G}_t
  (\omega)$ denote the smallest element of ${\mathscr G}_t$ containing
  $\omega$. Given $t \in T$, $\omega \in \Omega$ and a
  price system $\pi$, let $B^{REE}(t,\omega,\pi)$ be defined by
   \[
  B^{REE}(t, \omega, \pi)= \left\{x\in (\mathbb R^\ell_+)^\Omega:
  x(\omega^\prime)\in B(t, \omega^\prime, \pi(\omega^\prime))
  \mbox{ for all } \omega^\prime \in \mathscr G_t(\omega)\right\}.
  \]
  The \emph{maximin utility} of each agent $t\in T$ with
  respect to $\mathscr G_t$ at an allocation $f: T \times \Omega
  \to \mathbb R^\ell_+$ in state $\omega \in \Omega$, denoted
  by $\b{\it U}^{REE}(t, \omega, f(t,\cdot))$, is defined by
  \[
  \b{\it U}^{REE}(t,\omega, f(t,\cdot))= \inf_{\omega'\in
  \mathscr G_t (\omega)} U(t, \omega', f(t,\omega')).
  \]

  \begin{definition}\label{def:maximinree}
  Given a feasible allocation $f$ and a price system $\pi$, the
  pair $(f, \pi)$ is called a \emph{maximin rational expectations
  equilibrium} of $\mathscr E$ if $f(t, \omega)\in B(t, \omega,
  \pi(\omega))$ for almost all $(t, \omega) \in T \times \Omega$,
  and $f(t, \cdot)$ maximizes $\b{\it U}^{REE}(t, \omega, \cdot)$
  on $B^{REE}(t, \omega, \pi)$.
  \end{definition}

  An allocation $f$ is called a \emph{maximin rational expectations
  allocation} if there exists a price system $\pi$ such that $(f,
  \pi)$ is a maximin rational expectations equilibrium. The set of maximin rational expectations allocations is denoted
  by $MREE(\mathscr E)$. Recently, de Castro et al.
  \cite{de Castro-Pesce-Yannelis:11} showed that $MREE(\mathscr E)
  \ne \emptyset$ when $\Omega$ and $T$ are finite. Our next
  theorem extends their result to a more general case.

  \begin{theorem}\label{thm:Aumann}
  Under \emph{({\bf A$_1$})-({\bf A$_4$})}, $MREE(\mathscr E) \ne
  \emptyset$.
  \end{theorem}

  \begin{proof}
  Consider the correspondence $Z:\Omega\times \Delta
  \rightrightarrows \mathbb R^\ell$ defined by
  \[
  Z(\omega, p)= \int_T C^{X}(\cdot, \omega,
  p)d\mu- \int_T a(\cdot, \omega)d\mu.
  \]
  By Proposition \ref{prop:Clowermeasure}, $Z$ is non-empty
  compact-valued. In addition, by Corollary \ref{coro:jointmea}
  and ({\bf A$_1$}),
  $Z: \Omega\times \Delta \rightarrow \mathbb {\mathscr K}_0
  (\mathbb R^\ell)$ is jointly measurable. Define another
  correspondence $F: \Omega\rightrightarrows \Delta$ such
  that
  \[
  F(\omega)=\{p\in \Delta: Z(\omega, p)\cap \{0\}\neq
  \emptyset\}.
  \]
  By Theorem 2 in \cite[p.151]{Hildenbrand:74},
  $F$ is non-empty valued. Since
  \[
  {\rm Gr}_F=Z^{-1}(\{0\}),
  {\rm Gr}_{F} \in \mathscr F\otimes \mathfrak B(\Delta),
  \]
  by the Aumann-Saint-Beuve measurable selection theorem, there
  exists a measurable function $\hat{\pi}:\Omega\to \Delta$
  such that $\hat{\pi}(\omega)\in F (\omega)$ for all $\omega
  \in \Omega$. By the definition of $Z$, there exists an
  allocation $f$ such that $f(t, \omega)\in C^{X}(t,\omega,
  \hat{\pi}(\omega))$ and
  \[
  \int_T f(\cdot, \omega)d\mu= \int_T a(\cdot, \omega)d\mu
  \]
  for almost all $t\in T$ and all $\omega \in \Omega$.
  By Proposition \ref{prop:implication}, one has
  \[
  \langle \hat{\pi}(\omega), f(t,\omega)\rangle\ge \langle
  \hat{\pi}(\omega), a(t,\omega)\rangle
  \]
  for almost all $t\in T$ and all $\omega \in \Omega$. Then,
  the previous equation implies
  \[
  \langle \hat{\pi}(\omega), f(t,\omega)\rangle= \langle
  \hat{\pi}(\omega), a(t,\omega) \rangle
  \]
  for almost all $t\in T$ and all $\omega \in
  \Omega$. Thus, $f(t, \omega)\in B(t, \omega, \hat{\pi}
  (\omega))$ for almost all $t\in T$ and all $\omega\in
  \Omega$. For every $\omega\in \Omega$, define $T_\omega
  \subseteq T$ by
  \[
  T_\omega=\{t\in T: f(t, \omega) \in B(t, \omega,
  \hat{\pi}(\omega))\cap C(t,\omega, \hat{\pi}(\omega))
  \}.
  \]
  Then, $\mu(T_\omega)= \mu(T)$ for all $\omega \in \Omega$.
  Next, for every $\omega \in \Omega$ and every $t\in T
  \setminus T_\omega$, as $B(t, \omega, \hat{\pi} (\omega))
  \cap C(t, \omega, \hat{\pi} (\omega)) \ne \emptyset$, one
  can pick a point
  \[
  h(t, \omega)\in B(t, \omega, \hat{\pi} (\omega))\cap
  C(t, \omega, \hat{\pi}(\omega)),
  \]
  and then define a function $\hat f: T \times \Omega \to \mathbb
  R^\ell_+$ such that
  \[
  \hat f(t, \omega)= \left\{
  \begin{array}{ll}
  f(t, \omega), & \mbox{if $t\in T_\omega$;}
  \\[0.5em]
  h(t, \omega), & \mbox{if $t\in T\setminus T_\omega$.}
  \end{array}
  \right.
  \]
  It is obvious that $\hat f(t, \omega)\in B(t, \omega,
  \hat\pi(\omega)) \cap C(t,\omega, \hat\pi(\omega))$ for all
  $(t, \omega) \in T \times \Omega$.
  Assume that there are an agent $t_0 \in T$, a state of nature
  $\omega_{t_0}\in \Omega$ and an element $y(t_0, \cdot)
  \in B^{REE} (t_0, \omega_{t_0}, \hat{\pi})$ such that
  \[
  \b {\it U}^{REE}(t_0, \omega_{t_0}, y(t_0, \cdot))>
  \b{\it U}^{REE}(t_0, \omega_{t_0}, \hat f(t_0, \cdot)).
  \]
  Then, one obtains
  \[
  U(t_0, \omega_{t_0}^\prime, y(t_0, \omega_{t_0}^\prime))
  > U(t_0, \omega_{t_0}^\prime, \hat f(t_0, \omega_{t_0}^\prime))
  \]
  for some $\omega_{t_0}^\prime\in \mathscr G_{t_0} (\omega_{t_0})$,
  which contradicts with $\hat f(t_0,
  \omega_{t_0}^\prime)\in C(t_0, \omega_{t_0}^\prime,
  \hat{\pi}(\omega_{t_0}^\prime))$. This verifies that $(\hat f,
  \hat{\pi})$ is a maximin rational expectations equilibrium of
  $\mathscr E$.
  \end{proof}

  \section{Conclusion} \label{sec:conclusion}
  The first application of maximin expected utility functions to
  the general equilibrium theory with differential information appeared
  in Correia da Silva and Herv\'{e}s Beloso \cite{Correia-Herves:09},
  where an existence theorem for a Walrasian equilibrium in an
  economy was established. However, their MEU formulation is in
  the ex-ante sense, and REE notion was not considered.

  \medskip
  In our paper, an existence theorem on a maximin rational expectations
  equilibrium (maximin REE) for an exchange differential
  information economy is proved. Comparing with the existence result
  on maximin REE in \cite{de Castro-Pesce-Yannelis:11}, our theorem
  applies to a more general economic model with an arbitrary finite
  measure space of agents and an arbitrary complete probability
  measure space as the space of states of nature, while the later
  applies only to an economic model which has finitely many agents
  and finitely many states of nature. Assumptions in our paper are similar to those in \cite{de Castro-Pesce-Yannelis:11}, except
  the joint measurability and continuity of utility functions,
  and the joint measurability of the initial endowment function.
  The proof techniques in this paper are quite different from
  those in \cite{de Castro-Pesce-Yannelis:11}. Since there are
  only finitely many agents and states of nature in the model
  considered in \cite{de Castro-Pesce-Yannelis:11}, neither
  measurability nor continuity of utility functions and the
  initial endowment function plays any role in the proof of
  the existence of a maximin REE. Instead, the existence of a
  competitive equilibrium for complete information economies is
  applied. In contrast, both measurability and continuity of
  utility functions and the initial endowment function play
  key roles in this paper.
  To establish the existence theorem, Aumann's techniques in
  \cite{Aumann:66} are adopted, measurability and continuity of
  the aggregate preferred correspondence are investigated. However,
  for special cases, the techniques can be simplified. For instance,
  if there are finitely many states of nature, one can still apply
  the approach employed in \cite{de Castro-Pesce-Yannelis:11} and
  obtains an existence theorem. On the other hand, if there are
  finitely many agents, then one can show that the demand of each
  agent is ${\mathscr F} \otimes {\mathfrak B}(\Delta)$-measurable
  and so is the aggregate demand. Then, an approach similar to
  that in the proof of Theorem \ref{thm:Aumann} can be applied to
  establish the existence theorem.

  \bigskip
  \centerline{\large \sc Appendix}

  \bigskip
  Let $G$ be a non-empty set and $\mathbb R^\ell$ be the $\ell$-dimensional Euclidean space. On $\mathbb R^\ell$, two
  different but equivalent standard norms $\|\cdot\|_\infty$
  and $\|\cdot\|_1$ are used in this paper, where for each
  pint $x=(x_1, x_2, \cdots, x_\ell)\in \mathbb R^\ell$,
  \[
  \|x\|_\infty = \max \{|x_i|: 1\le i \le \ell\},
  \]
  and
  \[
  \|x\|_1 = \sum_{1\le i \le \ell} |x_i|.
  \]
  A \emph{correspondence} $F: G \rightrightarrows \mathbb R^\ell$
  from $G$ to $\mathbb R^\ell$ assigns to each  $x\in G$ a subset
  $F(x)$ of $\mathbb R^\ell$. Meanwhile, $F$ can also be viewed
  as a function $F: G \to 2^{\mathbb R^\ell}$, where $2^{\mathbb
  R^\ell}$ denotes the power set of $\mathbb R^\ell$. Further,
  $F$ is called non-empty valued (resp. closed-valued,
  compact-valued, convex-valued) if $F(x)$ is a non-empty (resp.
  closed, compact, convex) subset of $\mathbb R^\ell$ for all
  $x\in G$. The \emph{graph} of $F$, denoted by ${\rm Gr}_F$, is 
  defined by
  \[
  {\rm Gr}_F = \{ (x,y) \in G \times \mathbb R^\ell: y \in F(x)\
  {\rm and}\ x \in G\}.
  \]
  For each point $x \in {\mathbb R}^\ell$ and a subset $A \in
  2^{\mathbb R^\ell}\setminus \{\emptyset \}$, define
  \[
  {\rm dist} (x, A) = \inf \{d(x,y): y \in A \},
  \]
  where $d$ is the Euclidean metric on $\mathbb R^\ell$.
  Let ${\mathscr K}_0\left(\mathbb
  R^\ell\right)$ be the family of non-empty compact subsets of
  $\mathbb R^\ell$. Recall that the \emph{Hausdorff metric} $H$ on
  ${\mathscr K}_0\left(\mathbb R^\ell\right)$ is defined such that
  for any two $A, B \in {\mathscr K}_0\left(\mathbb R^\ell\right)$,
  \[
  H(A, B)= \max\left\{\sup_{a\in A}\ {\rm dist}(a, B), \
  \sup_{b\in B}\ {\rm dist}(b, A)\right\}.
  \]
  For equivalent definitions
  of $H$, refer to \cite{aubin:08}. The topology ${\mathscr T}_H$ on
  ${\mathscr K}_0 (\mathbb R^\ell)$, generated by $H$, is called
  the \emph{Hausdorff metric topology}. For a closed subset $M$
  of $\mathbb R^\ell$, ${\mathscr K}_0(M)$ and the Hausdorff
  metric $H$ on ${\mathscr K}_0(M)$ can be defined similarly.
  When $G$ is a topological space, a non-empty compact-valued correspondence $F: G \rightrightarrows \mathbb R^\ell$ is
  called \emph{Hausdorff continuous} if $F: G \to \left(
  {\mathscr K}_0\left(\mathbb R^\ell\right), {\mathscr T}_H
  \right)$ is continuous. This statement also holds when
  $\mathbb R^\ell$ is replaced by a closed subset $M$ of
  $\mathbb R^\ell$.

  \bigskip
  Let $\{A_n: n\ge 1\}$ be a sequence of non-empty subsets of $\mathbb
  R^\ell$. A point $x\in \mathbb R^\ell$ is called a \emph{limit
  point} of $\{A_n: n\ge 1\}$ if there exist $N\ge 1$ and points
  $x_n\in A_n$ for each $n \ge N$ such that $\{x_n: n \ge N\}$
  converges to $x$. The set of all limit points of $\{A_n: n\ge
  1\}$ is denoted by ${\rm Li}A_n$. Similarly, a point $x\in
  \mathbb R^\ell$ is called a \emph{cluster point} of $\{A_n: n
  \ge 1\}$ if there exist positive integers $n_1< n_2 <\cdots$
  and for each $k$ a point $x_k \in A_{n_k}$ such that $\{x_k: k
  \ge 1\}$ converges to $x$. The set of all cluster points of
  $\{A_n: n\ge 1\}$ is denoted by ${\rm Ls}A_n$.
  It is clear that ${\rm Li}A_n\subseteq {\rm Ls}A_n$. If
  ${\rm Ls}A_n\subseteq {\rm Li}A_n$,
  \[
  {\rm Li}A_n= {\rm Ls}A_n= A
  \]
  is called the \emph{limit} of the sequence
  $\{A_n: n\ge 1\}$. If $A$ and all $A_n'$s are closed and
  contained in a compact subset $M \subseteq \mathbb R^\ell$,
  then it is well known that
  \[
  {\rm Li}A_n={\rm Ls}A_n= A
  \]
  if and only if $\{A_n: n \ge 1\}$ converges to $A$ in Hausdorff
  metric topology on ${\mathscr K}_0(M)$, refer to \cite{aubin:08}.

  \bigskip
  Let $(T, \Sigma, \mu)$ be a measure space and $\{F_n: n\ge 1
  \}, F: (T, \Sigma, \mu) \rightrightarrows \mathbb R^\ell_+$ be
  correspondences. Recall that $F: (T,\Sigma, \mu) \rightrightarrows
  \mathbb R^\ell$ is said to be \emph{lower measurable}
  if
  \[
  F^{-1}(V)=\{t\in T: F(t) \cap V \ne \emptyset\} \in
  \Sigma
  \]
  for every open subset $V$ of $\mathbb R^\ell$.
  It is well known that
  a non-empty closed-valued correspondence $F: (T,\Sigma, \mu)
  \rightrightarrows \mathbb R^\ell$ is lower measurable if and
  only if there exists a sequence of measurable selections
  $\{f_n: n\ge 1\}$ of $F$ such that for all $t\in T$,
  \[
  F(t)= \overline {\{f_n(t): n\ge 1\}}.
  \]
  If all $F_n's$ are non-empty closed-valued and lower
  measurable and at least one of $F_n's$ is compact-valued,
  $\bigcap_{n\ge 1} F_n$ is lower measurable, refer to
  \cite{Himmelberg:75}. If all $F_n's$ are integrably bounded
  by the same function, then
  \[
  {\rm Ls} \int_T F_nd\mu \subseteq \int_T
  {\rm Ls} F_n d\mu,
  \]
  and
  \[
  \int_T {\rm Li} F_n  d\mu\subseteq {\rm Li} \int_T F_n d\mu.
  \]
  If $(S, \mathscr S, \nu)$ is another measure space and
  $f: T \times S \to \mathbb R^\ell$ is a jointly measurable
  function, then it is well known that $\int_T: f(\cdot,
  \cdot)d\mu: S \to \mathbb R^\ell$ is a measurable function.
  Let  $M \subseteq \mathbb R^\ell$ be endowed with the relative
  Euclidean topology, and $(Y,\varrho)$ be a metric space. A
  function $f: T\times M \to (Y,\varrho)$ is called
  \emph{Carath\'{e}odory} if $f(\cdot,x)$ is measurable for all
  $x\in M$, and $f(t, \cdot)$ is continuous for all $t\in T$.
  It is known that any Carath\'{e}odory function is jointly
  measurable with respect to the Borel structure on $M$.
  A  \emph{selection} of $F$ is a single-valued function $f:
  (T,\Sigma, \mu) \to \mathbb R^\ell$ such that $f(t)\in F(t)$
  for almost all $t\in T$. If a selection $f$ of $F$ is measurable
  (resp. integrable), then it is called a \emph{measurable} (resp.
  an \emph{integrable}) \emph{selection}. Let ${\mathscr S}_F$
  denote the set of all integrable selections of $F$. The
  \emph{integration} of $F$ over $T$ in the sense of Aumann
  in \cite{Aumann:65} is a subset of $\mathbb R^\ell$, defined as
  \[
  \int_T F d\mu =\left\{ \int_T fd\mu: f\in {\mathscr S}_F
  \right\}.
  \]
  If $F$ is non-empty closed-valued and integrably bounded, then
  $\int_T F d\mu$ is compact, refer to
  \cite{Hildenbrand:74}. The following two theorems on measurable
  selections have been employed in this paper.

  \bigskip
  \noindent
  {\bf Kuratowski-Ryll-Nardzewski Measurable Selection Theorem
  (\cite{Himmelberg:75,Kuratowski:65}).}
  \emph{If $F: (T,\Sigma,\mu) \rightrightarrows \mathbb R^\ell$ is a
  closed-valued and lower measurable correspondence, then it has
  a measurable selection.}

  \bigskip
  \noindent
  {\bf Aumann-Saint-Beuve Measurable Selection Theorem
  (\cite{Aumann:69,Saint:74}).} \emph{Let $(T,\Sigma,\mu)$
  be a complete finite measure space, $B \subseteq \mathbb R^\ell$
  be a Borel subset. If $F:T \rightrightarrows B$ has a measurable
  graph, there exists a measurable function $f: T \to B$ such that $f(t)\in F(t)$ for all $t\in T$".}

  \bigskip
  Of course, the above two theorems were stated in more general
  forms in the literature. But here, they are adapted and presented
  in particular and simpler forms to fit in this paper.

  \bigskip
  \noindent
  {\large \sc Acknowledgement.} The authors are very grateful to
  He Wei for his valuable comments and suggestions on the early
  draft of the paper.

%\begin{figure}
%\includegraphics{example.eps}
%\end{figure}
%\begin{figure*}
%\includegraphics[width=0.75\textwidth]{example.eps}
%\end{figure*}
%\begin{table}
%\begin{tabular}{lll}
%\hline\noalign{\smallskip}
%\end{tabular}
%\end{table}

%\begin{acknowledgements}
%\end{acknowledgements}

\bibliographystyle{spbasic}      % basic style, author-year citations

\begin{thebibliography}{00}

  \bibitem{Allen:81}
  Allen, B.: Generic existence of completely revealing equilibria
  with uncertainty, when prices convey information, Econometrica
  {\bf 49}, 1173--1199 (1981)

  \bibitem{aubin:08}
  Aubin, J.P., Frankowska, H.: Set-valued analysiss. Springer (2008).

  \bibitem{Aumann:65}
  Aumann, Robert J.: Integrals of set-valued functions, J. Math.
  Anal. Appl. {\bf 12}, 1-–12 (1965).

  \bibitem{Aumann:66}
  Aumann, R. J.: Existence of competitive equilibria in markets
  with a continuum of traders, Econometrica {\bf 34}, 1--17 (1966)

  \bibitem{Aumann:69}
  Aumann, R. J.: Measurable utility and the measurable choice theorem,
  La D\'{e}cision Colloque Internationaux du C.N.R.S., Paris
  {\bf 171}, 15--26 (1969)

  %\bibitem{Condie-Ganguli:11}
  %Condie, S., Ganguli, J.~V.: Informational efficiency with ambiguous
  %information, Econ Theory {\bf 48}, 229--242 (2011)

  \bibitem{Correia-Herves:09}
  Correia-da-Silva, J., Herv\'{e}s-Beloso, C.: Prudent expectations
  equilibrium in economies with uncertain delivery, Econ. Theory
  {\bf 39}, 67--92 (2009)

  %\bibitem{Correia-Herves:12}
  %Correia-da-Silva, J., Herv\'{e}s-Beloso, C.: General equilibrium
  %in economies with uncertain delivery, Econ Theory (2012), forthcoming

  %\bibitem{de Castro-Pesce-Yannelis:11}
  %De Castro, L.~I., Pesce, M., Yannelis, N.~C.: Core and equilibria
  %under ambiguity, Econ Theory {\bf 48}, 519--548 (2011)

  \bibitem{de Castro-Pesce-Yannelis:11}
  de Castro, L.~I., Pesce, M., Yannelis, N.~C.: A new perspective to
  rational expectations: maximin rational expectation equilibrium,
  under preparation

  \bibitem{Gilboa:89}
  Gilboa, I., Schmeidler, D.: Maxmin expected utility with a
  non-unique prior, J. Math. Econ. {\bf 18}, 141--153 (1989)

  \bibitem{Glycopantis:2005}
  Glycopantis, D., Muir, A., Yannelis, N.C.: Non-implementation
  of rational expectations as a perfect Bayesian equilibrium, Econ.
  Theory {\bf 26}, 765--791 (2005)

  \bibitem{Hildenbrand:74}
  Hildenbrand, W.: Core and equilibria in large economies, Priceton
  University Press, 1974

  \bibitem{Himmelberg:75}
  Himmelberg, C.~J.: Measurable relations, Fund. Math. {\bf 87},
  53–-72 (1975)

  \bibitem{Kreps:77}
  Kreps, D.~M.: A note on `fulfilled expectations' equilibrium,
  J. Econ. Theory {\bf 14(1)}, 32--43 (1977)

  \bibitem{Kuratowski:65}
  Kuratowski, K., Ryll-Nardzewski, C.: A general theorem on selectors,
  Bull. Acad. Polon. Sci. {\bf 13}, 397--403 (1965)

  %\bibitem{Ozsoylev-Werner:59}
  %Ozsoylev, H., Werner, J.: Liquidity and asset prices in rational
  %expectations equilibrium with ambiguous information, Econ Theory
  %{\bf 48}, 469--491 (2011)

  \bibitem{Radner:79}
  Radner, R.: Rational expectation equilibrium: generic existence
  and information revealed by prices, Econometrica {\bf 47},
  655--678 (1979)

  \bibitem{Saint:74}
  Saint-Beuve, M.-F.: On the extension of von Neumann-Aumann's
  theorem, J. Funct. Anal. {\bf 17}, 112--129 (1974)

  \bibitem{Yosida:80}
  Yosida, K.: Functional analysis, Sixth edition. Springer (1980).

% etc
\end{thebibliography}


 \end{document}